\documentclass{mathscan}
\usepackage{hyperref}

\newtheorem{thm}{Theorem}[section]
\newtheorem{lem}[thm]{Lemma}
\newtheorem{prop}[thm]{Proposition}
\newtheorem{cor}[thm]{Corollary}
\newtheorem*{conj}{Conjecture}
\newtheorem*{thmKhint}{Theorem (Khinchine)}
\newtheorem*{thmDS}{Duffin--Schaeffer Theorem}
\newtheorem*{thmHayn}{$p$-adic Duffin--Schaeffer Theorem}
\newcommand{\DirThm}{Dirichlet's Theorem on primes}
\newtheorem*{thmDir}{\DirThm}

\numberwithin{equation}{section}

\newcounter{case}
\newcommand{\case}[1]{\refstepcounter{case}\textbf{Case \thecase: #1.}}

\newcommand{\bN}{\mathbb{N}}
\newcommand{\bR}{\mathbb{R}}
\newcommand{\bQ}{\mathbb{Q}}
\newcommand{\bZ}{\mathbb{Z}}
\newcommand{\cA}{\mathcal{A}}
\newcommand{\cB}{\mathcal{B}}
\newcommand{\cC}{\mathcal{C}}
\newcommand{\cK}{\mathcal{K}}
\newcommand{\fA}{\mathfrak{A}}
\newcommand{\fK}{\mathfrak{K}}

\DeclareMathOperator{\supp}{supp}
\newcommand{\linebreakLong}{\\[0.15cm]}

\begin{document}
	
	\title[Attainable measures of p-adic Duffin--Schaeffer sets]{Attainable measures for certain types of p-adic Duffin--Schaeffer sets}
	\author{MATHIAS L. LAURSEN}
	\address{Department of Mathematics \\
		Aarhus University \\
		8000 Aarhus C, \\
		Denmark}
	\email{mll@math.au.dk}

	\begin{abstract}
		This paper settles recent conjectures concerning the $p$-adic Haar measure applied to a family of sets defined in terms of Diophantine approximation.
		This is done by determining the spectrum of measure values for each family and seeing that this contradicts the corresponding conjectures.
	\end{abstract}

	\maketitle
	
	
	\section{Introduction}
	In Diophantine approximation, we are often interested in determining when a given number $\alpha\in[0,1]$ has infinitely many solutions to inequalities of the form
	\begin{equation}\label{eq:approx}
		 \left \vert \alpha - \frac{a}{n} \right\vert < \frac{\psi(n)}{n},
	\end{equation}
	for some chosen function $\psi:\bN\to\bR_{\ge 0}$.
	In 1924, Khintchine \cite{Khinchin} gave one of the first metric results regarding this inequality,
	as stated below.
\begin{thmKhint}
	Let $\psi:\bN\to\bR_{\ge 0}$ and write
	\begin{align*}
		\cK(\psi) = \left\{ \alpha \in [0,1] : \text{\eqref{eq:approx} holds for infinitely many } (a,n) \in \bZ\times\bN \right\}.
	\end{align*}
	If $\psi(n)$ is monotonic, then $\cK(\psi)$ has Lebesgue measure 1 if $\sum_{n=1}^\infty \psi(n) = \infty$ and has Lebesgue measure 0 otherwise.
\end{thmKhint}

	Later, Duffin and Schaeffer \cite{Duffin+Schaeffer} tried to remove the monotonicity condition and found a counterexample, which lead them to suppose that it would be more natural to consider the set
	\begin{equation}
	\label{eq:main}
	\cA(\psi) = \left\{ \alpha \in [0,1] : \text{\eqref{eq:approx} holds for infinitely many } \frac{a}{n} \in \bQ \right\},
	\end{equation}
	where the fractions $a/n$ are assumed to be reduced, i.e. $\gcd(a,n)=1$.
	They then famously conjectured the below theorem, which was only proven a couple of years ago by Koukoulopoulos and Maynard \cite{Koukoulopoulos+Maynard}.
	\begin{thmDS}
		Let $\psi:\bN\to\bR_{\ge 0}$. Then
		\begin{equation*}
			\vert \cA(\psi) \vert = 
			\begin{cases} 
				0 & \text{if } \sum_{n=1}^\infty \frac{\psi(n) \phi(n)}{n} < \infty, \\
				1 & \text{if } \sum_{n=1}^\infty \frac{\psi(n) \phi(n)}{n} =\infty.
			\end{cases}
		\end{equation*}
	\end{thmDS}
	Here and throughout this paper, $\phi$ denotes the Euler totient function.
	Named after the conjecture, we will refer to $\cA(\psi)$ as the (real) Duffin--Schaeffer set.
	In handling the conjecture, one usually writes it as a lim-sup set $\cA(\psi)= \limsup_{n\to\infty} \cA_n(\psi)$ where
	\begin{align*}
		\cA_n(\psi) = \bigcup_{\substack{1\le a\le n \\ \gcd(a,n)=1}} B_\bR \bigg(\frac{a}{n}, \frac{\psi(n)}{n}\bigg)\cap[0,1].
	\end{align*}
	We use $B_X(x,r)$ to denote the closed ball of radius $r$ around the point $x$ in metric space $X$ when $r>0$, and the singleton $\{x\}$ when $r=0$.
	
	Before the conjecture was settled, Haynes \cite{Haynes} proposed a $p$-adic variant of the conjecture in terms of the $p$-adic Haar measure $\mu_p$ with $\mu_p(\bZ_p)=1$ and a $p$-adic Duffin--Schaeffer set, $\cA^p$.
	To this end, he first translated $\cA_n$ into
	\begin{align*}
		\cA^p_n(\psi) := \bigcup_{\substack{|a|\le n \\ \gcd(a,n)=1}} B_{\bQ_p} \bigg(\frac{a}{n}, \frac{\psi(n)}{n}\bigg)\cap\bZ_p
	\end{align*}
	and then defined $\cA^p$ as the limsup set of $\cA^p_n(\psi)$.
	From this, Haynes phrased a $p$-adic Duffin--Schaeffer conjecture for the set $\cA^p(\psi)$.
	One of the main results of Haynes' paper was that if a certain `quasi-independentness on average' criterion, closely related to the real Duffin--Schaeffer conjecture, were to hold, then his conjecture would follow.
	In \cite{Kristensen+Laursen}, Kristensen and Laursen modified arguments from \cite{Koukoulopoulos+Maynard} to prove this `quasi-independentness on average' for relevant $\psi$, thus settling the conjecture in the affirmative, so we may now state it as a theorem.
	\begin{thmHayn}
		We have
		\begin{equation*}
			\displaystyle\mu_p(\cA^p(\psi)) = \begin{cases}
				0	&\text{if } \sum_{n=1}^\infty \mu_p(\cA^p_n(\psi)) < \infty	\\
				1	&\text{if } \sum_{n=1}^\infty \mu_p(\cA^p_n(\psi)) = \infty.
			\end{cases}
		\end{equation*}
	\end{thmHayn}
	
	However, the set $\cA^p(\psi)$ is not the only natural choice for a $p$-adic variant of $\cA(\psi)$, as one might instead start from equation \eqref{eq:main} when translating the question to a $p$-adic context.
	For this, we first need a translation of inequality \eqref{eq:approx}.
	Following Jarník \cite{Jarnik} and Lutz \cite{Lutz}, this should be
	\begin{equation}
	\label{eq:approx p-adic}
	\left\vert \alpha - \frac{a}{n} \right\vert_p \le \frac{\psi(\max\{|a|,|n|\})}{\max\{|a|,|n|\}}.
	\end{equation}
	By fixing $\max\{|a|,|n|\}$, we limit both the numerator and the denominator.
	If we, as is done in the real case, merely compared based on $n$, the inequality would trivially have infinitely many answers $a\in\bZ$ for any $x\in\bZ_p$ whenever $\psi(n)>0$, as the fractions $a/n$ would be dense in the ball of radius $|n|_p^{-1}$ around the origin, by virtue of $\bZ$ being dense in $\bZ_p$.
	If we impose the condition that $a/n$ be a reduced fraction, the argument becomes slightly more obscure, but the fractions would still be dense in $\bZ^p$ when $\gcd(a,n)=1$ -- just add the proper multiple of a sufficiently large power of $p$ to $a$ when $\gcd(a,n)>1$.
	We then define the alternative $p$-adic Duffin--Schaeffer set as
	\begin{equation*}
		\cB^p(\psi) = \left\{ \alpha \in \bZ_p : \eqref{eq:approx p-adic} \text{ holds for inifnitely many } \frac{a}{n} \in \bQ \right\}.
	\end{equation*}
	This set was also briefly considered by Haynes \cite{Haynes}, who showed that this set allowed for $\mu_p(\cB^p(\psi))\ne 0,1$, by explicitly constructing a $\psi$ such that the measure was $p^{-1}$.
	Kristensen and Laursen \cite{Kristensen+Laursen} revisited the set and found an uncountable collection of possible values for $\mu_p(\cB^p(\psi))$ when $\psi$ varies.
	For $p=2$, they found that the possible values were exactly those in the set $[0,1/2]\cup\{1\}$ and conjectured that the set of possible values would be $[0,1/p]\cup\{1\}$ for $p>2$.
	The first main result of this paper rejects this conjecture by proving that the spectrum of possible values of $\mu_p(\cB^p(\psi))$ is in fact limited to those already found in \cite{Kristensen+Laursen}.
	As noted in that paper, this means that the spectrum becomes a Lebesgue-null set of Hausdorff dimension $\log 2 / \log p$.

	Furthermore, a new paper by Badziahin and Bugeaud \cite{Badziahin+Bugeaud} introduces another Duffin--Schaeffer-like set of $p$-adic numbers, which we will denote by ${\fA'}^p(\psi)$.
	The definition of this set is related to that of $\cB^p(\psi)$ where inequality \eqref{eq:approx p-adic} is replaced by 
	\begin{equation}
		\label{eq:approx p-adic var var}
		\left\vert \alpha - \frac{a}{b} \right\vert_p < \frac{\psi(|ab|)}{|ab|},
	\end{equation}
	so that ${\fA'}^p(\psi)$ becomes
	\begin{equation*}
		{\fA'}^p(\psi) = \left\{ \alpha \in \bZ_p : \eqref{eq:approx p-adic var var} \text{ holds for inifnitely many } \frac{a}{b} \in \bQ \right\}.
	\end{equation*}
	This set is not quite a generalisation of the original Duffin--Schaeffer set, but it appears to be of a similar nature.
	In their paper, Badziahin and Bugeaud prove a theorem inspired by the above theorem due to Khintchine and by the Duffin--Schaeffer Theorem.
	Akin to Khintchine, they achieve their result by imposing certain restrictions on the approximation function $\psi$, including some weak growth restrictions.
	After stating the theorem, they suggest that it may hold with these restrictions weakened or perhaps even removed.
	We will in this paper show that said restrictions cannot be removed entirely as the spectrum of values attainable by $\mu_p({\fA'}^p(\psi))$ 
	is the entire unit interval when $\psi$ is allowed to be any function onto $\bR_{\ge 0}$.	
	 
\section{Main Results}\label{section:main}
	For ease of notation, we will suppress $\psi$ when talking about the various Duffin--Schaeffer-inspired sets in the cases where there is no ambiguity towards the underlying $\psi$.
	
	The most obvious reason for the difference between $\cA^p$ and $\cB^p$ is that $\cA^p$ only allows fractions where the numerator is bounded by the denominator while $\cB^p$ allows numerators of any size relative to the denominator.
	As such, it appears natural to consider what happens with the remaining fractions, where the numerator is greater than the denominator.
	This leads to a third $p$-adic Duffin--Schaeffer set, $\cC^p$, which we define as the limsup set of sets $\cC^p_n$ given by
	\begin{equation*}
		\cC^p_n = \bigcup_{\substack{|a|< n \\ \gcd(a,n)=1	}}	B_{\bQ_p}\bigg(\frac{n}{a},\frac{\psi(n)}{n}\bigg)\cap \bZ_p,
	\end{equation*}
	for $n>1$, and $\cC^p_1 =\emptyset$ for $n=1$.
	It is then easy to see that $\cB^p$ is the limsup set of $\cB^p_n:=\cA^p_n \cup \cC^p_n$.
	By Dirichlet's pigeonhole principle, we then get
	\begin{equation}\label{eq:Bp=ApcupCp}
		\cB^p = \cA^p \cup \cC^p.
	\end{equation}
	
	If $p\nmid n$ and $p\mid a$, then $|n/a|_p > 1$, so that $B(n/a, \psi(n)/n)\cap \bZ_p = \emptyset$ unless $\psi(n)/n \ge |a|_p^{-1}>1$, in which case $\cC^p_n = \bZ_p = B(n/1, \psi(n)/n)\cap \bZ_p$.
	By this realization, we may instead write
	\begin{equation*}
	\cC^p_n(\psi) = \bigcup_{\substack{|a|< n \\ \gcd(a,pn)=1	}}	B_{\bQ_p}\bigg(\frac{n}{a},\frac{\psi(n)}{n}\bigg)\cap \bZ_p,
	\end{equation*}
	Note that this is also applicable for $n=1$, as the union is empty in that case.
	We will therefore use this as the \textit{de facto} definition of $\cC^p_n(\psi)$ going forward.
	
	We are now ready to state the first main result of this paper, which determines the spectra of values for $\mu_p(\cC^p)$ and $\mu_p(\cB^p)$, respectively.
	\begin{thm}
		\label{thm:spectrum}
		Let $x\in[0,1]$.
		There exists a function $\psi:\bN\to\bR_{\ge 0}$ such that $\mu_p(\cC^p(\psi)) = x$ if, and only if, $x$ is of the form
		\begin{align*}
			x = \sum_{k=0}^{\infty} x_k (p-1) p^{-k-1},
			\quad x_k\in\{0,1\}\ \forall k\in\bN_0.
		\end{align*}
		There exists a function $\psi:\bN\to\bR_{\ge 0}$ such that $\mu_p(\cB^p(\psi)) = x$ if and only if, $x=1$ or $x$ is of the form
		\begin{align*}
			x = \sum_{k=1}^{\infty} x_k (p-1) p^{-k-1},
			\quad x_k\in\{0,1\}\ \forall k\in\bN
			.
		\end{align*}
	\end{thm}
	
	We now turn our attention towards the set ${\fA'}^p$ from \cite{Badziahin+Bugeaud}.
	In order to have notation more in line with the Duffin--Schaeffer sets $\cA$ and $\cB^p$, we alter inequality \eqref{eq:approx p-adic var var} slightly,
		\begin{equation}
		\label{eq:approx p-adic var}
		\left\vert \alpha - \frac{a}{b} \right\vert_p \le \frac{\psi(|ab|)}{|ab|},
	\end{equation}
	which leads to a set
	\begin{equation*}
		\fA^p(\psi) = \left\{ \alpha \in \bZ_p : \eqref{eq:approx p-adic var} \text{ holds for inifnitely many } \frac{a}{b} \in \bQ \right\}.
	\end{equation*}
	To see that this set will have the same spectrum of measures as ${\fA'}^p$, define $\psi'$ by
	\begin{align*}
		\psi'(n) = \begin{cases}
			p\psi(n)	&\text{if } \psi(n)=n/p^k	\text{ for some } k\in\bZ,	\\
			\psi(n)	&\text{otherwise}.
		\end{cases}
	\end{align*}
	Note that
	\begin{equation*}
		\left|\alpha - \frac{a}{n}\right|_p \le 0
	\end{equation*}
	has at most 1 solution $a/n\in \bQ$ for any given $\alpha$, and so the collection of $n$ with $\psi(n)=0$ contributes to neither $\fA^p(\psi)$ nor ${\fA'}^p(\psi')$.
	For $\psi(n)>0$, inequality \eqref{eq:approx p-adic var var} with $\psi'$ is equivalent to inequality \eqref{eq:approx p-adic var} with $\psi$. 
	It thus follows that $\fA^p(\psi) = {\fA'}^p(\psi')$.
	Since there is an obvious bijection between $\psi$ and $\psi'$, $\fA^p$ and ${\fA'}^p$ have the same spectrum of measures.	
	
	We will also consider a Khintchine-like variant of the set, defined as
	\begin{equation*}
		\fK^p(\psi) = \left\{ \alpha \in \bZ_p : \eqref{eq:approx p-adic var} \text{ holds for inifnitely many } (a,b) \in (\bZ\setminus\{0\})\times \bN \right\}.
	\end{equation*}
	By taking $\bZ_\infty$, $\bQ_\infty$, $|\cdot|_\infty$, and $\mu_\infty$ to mean $[0,1]$, $\bR$, $|\cdot|$, and the Lebesque measure, respectively, we define related sets over the real numbers when allowing $p=\infty$.

	As is the case for the `proper' Duffin--Schaeffer sets, we will want to write $\fA^p$ and $\fK^p$ as the limsup set of sets $\fA^p_n$ and $\fK^p_n$, respectively.
	Write $B_{\bQ_p}(\pm a, r)=B_{\bQ_p}(a, r)\cup B_{\bQ_p}(-a, r)$ and $n=\prod_{i=1}^{\omega(n)} p_i^{\nu_{p_i}(n)}$, where $\omega(n)$ denotes the number of distinct prime divisors of $n$.
	By defining
	\begin{align*}
		\fK^p_n(\psi) &= \bigcup_{a\mid n} B_{\bQ_p}\bigg(\pm \frac{a}{n/a},\frac{\psi(n)}{n}\bigg)\cap\bZ_p,
		\\
		\fA^p_n(\psi) &= \bigcup_{a_1,\ldots,a_{\omega(n)}\in\{0,1\}} B_{\bQ_p}\bigg(\pm \frac{\prod_{a_i=1} p_i^{\nu_{p_i}(n)}}{\prod_{a_i=0} p_i^{\nu_{p_i}(n)}}, \frac{\psi(n)}{n}\bigg)\cap\bZ_p,
	\end{align*}
	we immediately achieve
	\begin{equation*}
		\fK^p = \limsup_{n\to\infty} \fK^p_n,
		\qquad
		\fA^p = \limsup_{n\to\infty} \fA^p_n.
	\end{equation*}
	It so happens that the spectra of values for $|\fA^\infty|$ and $|\fK^\infty|$ are surprisingly easy to settle while the spectra of values for $\mu_p(\fA^p)$ and $\mu_p(\fK^p)$ when $p<\infty$ require significantly more care.
	However, the spectra are nonetheless independent of $p$ as they always take up the entire unit interval as seen by the below theorem.
	We will prove the easy case of $p=\infty$ immediately and save $p<\infty$ for Section \ref{section:mult Ap}. That section will be independent of Sections \ref{section:p-adic measure theory} and \ref{section:spectrum Bp}.
	\begin{thm}
		\label{thm:spectrum var}
		Let $p$ be a prime or $p=\infty$, and let $x\in[0,1]$.
		Then there exists a $\psi:\bN\to\bR_{\ge 0}$ such that $\mu_p(\fA^p(\psi))=\mu_p(\fK^p(\psi))=x$.
	\end{thm}
	\begin{proof}[Proof for $p=\infty$]
		Let $x\in [0,1]$ and define $\psi:\bN\to\bR_{\ge 0}$ by
		\begin{align*}
			\psi(n) = \begin{cases}
				nx	&\text{if } n\text{ is a prime}	\\
				0	&\text{otherwise} 
			\end{cases}
		\end{align*}
		Since $\psi$ is supported on the primes, we immediately have that $\fA^\infty=\fK^\infty$.
		We are thus left with showing that $|\fA^\infty|=x$.
		Upon applying the Borel-Cantelli Lemma, this follows by a brief calculation, where we use $q$ to denote primes:
		\begin{align*}
			|\fA^\infty| &= \Bigg|\limsup_{q\to\infty}\bigcup_{e_1,e_2\in\{-1,1\}} B_{\bR}(e_1 q^{e_2},x) \cap[0,1]\Bigg|
			\\&
			= \bigg|\limsup_{q\to\infty} [0,x+q^{-1}]\bigg|
			= x.
		\end{align*}
	This completes the proof
	\end{proof}

	Note that $\fA^p_n(\psi)\subseteq\fK^p_n(\psi)$ will hold with strict inclusion for all small enough $\psi$, so we should not expect equality between $\mu_p(\fA^p)$ and $\mu_p(\fK^p)$ in general.
	In fact, if we in the above proof had taken
	\begin{align*}
		\psi(n) = \begin{cases}
		nx	&\text{if } n\text{ is the square of a prime},	\\
		0	&\text{otherwise},
		\end{cases}
	\end{align*}
	we would still find $|\fA^\infty| =x$ while $|\fK^\infty| = \min\{2x, 1\}$ as $\fK^p$ now also accepts infinitely many copies of $1=q/q$ as approximants and thereby includes the interval $[1-x,1]$ in $\fK^\infty$.
	
	\section{Some $p$-adic measure theory}\label{section:p-adic measure theory}
	In this section, we will present a series of general measure theoretical results that will be used for the proof of Theorem \ref{thm:spectrum}.
	Except for Lemma \ref{lem:Set inversion}, these results all appear to have been applied to some extent in \cite{Haynes}, even though only Lemma \ref{Lemma Gallagher} was formally stated.
	All results are to be applied in proving Proposition \ref{Prop:0-1 law}, which will be introduced in Section \ref{section:spectrum Bp} and plays a central role in proving the `only if' parts of Theorem \ref{thm:spectrum}.
	
	Note that each non-negative real number $x\in\bR_{\ge 0}$ has a canonical base $p$ expansion of the form
	\begin{align*}
	x=\sum_{n=N}^{\infty} a_n p^{-n},
	\end{align*}
	where $a_n\in \{0,\ldots,p-1\}$, $\liminf a_n < p-1$, and $N\in \bZ_{\le 0}$ is maximal possible.
	Throughout this paper, the function $\iota_p:\bR_{\ge 0}\to \bQ_p$ denotes the associated map
	\begin{align*}
		x=\sum_{n=N}^{\infty} a_n p^{-n}\mapsto \iota_p(x) = \sum_{n=N}^{\infty} a_n p^n.
	\end{align*}
	This function is measure preserving in that $\lambda(A)=p\mu_p(\iota_p(A))$ for any measurable subset $A\subseteq \bR$.
	\begin{lem}
		\label{Lemma lambda and mup}
		The preimage map $\iota_p^{-1}$ maps balls $B\subseteq \bQ_p$ to half-open intervals in $\bR$ of length $p\mu_p(B)$.
		In particular, $\iota_p$ is measurable, and the associated push-forward measure on the Lebesgue measure, $\iota_{p\#}$, is equal to $p\mu_p$.
	\end{lem}
	\begin{proof}
		Notice that
		\begin{align*}
		\iota_p^{-1}(\bZ_p) &= \bigg\{\sum_{m=0}^\infty b_m p^{-m} \in\bR:  b_m\in\{0,\ldots,p-1\},\ \liminf_{m\to\infty} b_m < p-1 \bigg\}
		\\&
		= [0,p).
		\end{align*}
		Let $B$ be a ball in $\bZ_p$ and write $B = \sum_{m=0}^{M-1} b_m p^{m} + p^M \bZ_p $ for some $M\in \bN_0$ and $b_0,\ldots, b_{M-1}\in\{0,\dots, p-1\}$.
		We then have
		\begin{align*}
			\iota_p^{-1}(B) &= \iota_p^{-1}\bigg( \sum_{m=0}^{M-1} b_m p^{m} \bigg) + p^{-M} \iota_p^{-1}(\bZ_p)	
			= \sum_{m=0}^{M-1} b_m p^{-m} + [0,p^{1-M}),
		\end{align*}
		which proves the first part of the lemma.
		Since the Borel algebra of $\bZ_p$ is generated by the collection of balls in $\bZ_p$, and the intervals are contained in the Borel algebra of $\bR_{\geq 0}$, $\iota_p$ is a measurable function, and the push-forward $\iota_{p\#}$ is a Borel-measure on $\bQ_p$.
		From the above equation, we notice in particular that for all balls $B$ in $\bQ_p$,
		\begin{align*}
		\iota_{p\#}(B) = p\mu_p(B).
		\end{align*}
		Since all balls have finite measure, and the collection of the empty set and all balls in $\bQ_p$ is preserved under pairwise intersection and generates the Borel algebra of $\bQ_p$, this implies that $\iota_{p\#} = p\mu_p$, by the Uniqueness of Measures Theorem, and the proof is complete.
	\end{proof}
	The alternative definition of $\mu_p$ provided by the above lemma gives a tool to translate measure theoretic results regarding $\bR$ into measure theoretic results regarding $\bQ_p$.
	It appears that Haynes may have used this alternative definition in \cite{Haynes}, though it was not explicitly stated.
	One result that very easily translates using $\iota_p$ is the below lemma from \cite{Gallagher}, which is a modification of a lemma from \cite{Cassels_1950}.
	\begin{lem}
		\label{Lemma Gallagher}
		Let $\{I_n\}_{n\in\bN}$ be a sequence of real intervals with $\lambda(I_n)\underset{n\to\infty}{\longrightarrow}0$, and let $\{U_n\}_{n\in\bN}$ be a sequence of measurable sets such that
		\begin{align*}
		U_n\subseteq I_n,
		\quad
		|U_n|\geq \varepsilon|I_n|
		\qquad
		\forall n\in\bN,
		\end{align*}
		for some fixed $0<\varepsilon<1$.
		Then $|\limsup_{n\to\infty} U_n| = |\limsup_{n\to\infty} I_n|$.
	\end{lem}
	\begin{cor}\label{Cor Gallagher}
		Let $\{B_n\}_{n\in\bN}$ be a sequence of $p$-adic balls with $\mu_p(B_n)\underset{n\to\infty}{\longrightarrow}0$.
		Suppose $\{U_n\}_{n\in\bN}$ is a sequence of measurable sets such that, for some
		\begin{align*}
		U_n\subseteq B_n,
		\quad
		\mu_p(U_n)\geq \varepsilon\mu_p(B_n)
		\qquad
		\forall n\in\bN,
		\end{align*}
		for some fixed $0<\varepsilon<1$.
		Then $\mu_p(\limsup_{n\to\infty} U_n)=\mu_p(\limsup_{n\to\infty} B_n)$.
	\end{cor}
	\begin{proof}
		Applying $p\mu_p = \iota_{p\#}$ by Lemma \ref{Lemma lambda and mup}, we have 
		\begin{align*}
		p\mu_p\Big(\limsup_{n\to\infty} U_n\Big)
		&= 
		\left\rvert \iota_p^{-1} \Bigg(\bigcap_{N=1}^\infty \bigcup_{n\ge N} U_n \Bigg)\right\rvert
		= \lim_{N\to\infty} \left\lvert \iota_p^{-1} \Bigg(	\bigcup_{n\ge N} U_n	\Bigg)\right\rvert
		\\&
		= \Big|\limsup_{n\to\infty} \iota_p^{-1}(U_n) \Big|.
		\end{align*}
		By repeating the process for $B_n$, the statement follows by Lemma \ref{Lemma Gallagher}, and the proof is complete.
	\end{proof}
	Another relevant result that may be derived using $\iota_p$ is a $p$-adic version of the Lebesgue Density Theorem, as presented below.
	\begin{lem}[Lebesgue Density Theorem for $\bQ_p$]
		\label{Lemma padic density theorem}
		Let $A$ be a measurable subset of $\bQ_p$ such that $\mu_p(A)>0$.
		Then $A$ contains a point of density 1, which is to say that
		there exists $a\in A$ such that if $\{B_M\}_{M\in\bN}$ is a sequence of balls $B_M\ni a$ of radius $r(B_M)\underset{M\to \infty}{\longrightarrow} 0$, then
		\begin{align*}
		\frac{\mu_p(A\cap B_M)}{\mu_p(B_M)}\underset{M\to \infty}{\longrightarrow} 1.
		\end{align*}
	\end{lem}
	\begin{proof}
		By Lemma \ref{Lemma lambda and mup}, we have $p \mu_p = \iota_{p\#}$.
		Thus $\tilde{A}=\iota_p^{-1}(A)$ has positive Lebesgue measure, which means that it must contain a point $x\in\tilde{A}$ of density 1, by the Lebesgue Density Theorem.
		Put $a=\iota_p(x)$, and let $\{B_M\}$ be a collection of balls around $a$ in $\bQ_p$ of radius $r(B_M)\underset{M\to \infty}{\longrightarrow} 0$. 
		
		Let $\varepsilon>0$, and put $\tilde{B}_M = \iota_p^{-1}(B_M)$.
		Since $\iota_p^{-1}$ maps balls to half-open intervals of length $p \mu_p(B_M)$ by Lemma \ref{Lemma lambda and mup}, and we must have $\tilde{B}_M\ni x$, it follows that $|\tilde{A}\cap \tilde{B}_M|>0$, as $x$ is of density 1 in $\tilde{A}$.
		Using once more that $\tilde{B}$ is an interval, this allows us to pick open intervals $I_M\supseteq\tilde{B}_M$ such that $|\tilde{A}\cap I_M|\le p^{\varepsilon} |\tilde{A}\cap \tilde{B}_M|$ and $|I_M| \le 2 |\tilde{B}_M|$.
		Hence,
		\begin{align}\label{eq:density calc}
		\frac{\mu_p(A\cap B_M)}{\mu_p (B_M)}
		= \frac{|\tilde A\cap \tilde B_M|}{|\tilde B_M|}
		\ge 
		 \frac{ p^{-\varepsilon} |\tilde{A}\cap I_M| }{ |I_M| }.
		\end{align}
		As $I_M\ni x$, write $I_M = (x-s_M, x+t_M)$ with $s_M,t_M>0$ and put $u_M = \max\{s_M, t_M\}$.
		Then
		\begin{align*}
			1\ge \frac{ |\tilde{A}\cap I_M| }{ |I_M| }  &
			\ge  \frac{ |\tilde{A}\cap (x-u_M,x+u_M)| - (u_M -s_M) - (u_M-t_M) }{ s_M+t_M } 
			\\&
			= \frac{2u_M}{s_M+t_M}  \frac{ |\tilde{A}\cap (x-u_M,x+u_M)| }{ |(x-u_M,x+u_M)| } - \frac{2u_M}{s_M+t_M} + 1.
		\end{align*}
		Since $x$ is of density 1, and we have 
		\begin{align*}
			u_M \le 2 |I_M| \le 4|\tilde{B}_M| = 4 p\, \mu_p(B_M) \underset{M\to \infty}{\longrightarrow} 0,
		\end{align*}
		the squeezing lemma then implies that $|\tilde{A}\cap I_M|/|I_M|\to 1$ as $M\to\infty$.
		Combined with equation \eqref{eq:density calc}, this means that
		\begin{align*}
			\liminf_{M\to\infty} \frac{\mu_p(A\cap B_M)}{\mu_p (B_M)}
			\ge p^{-\varepsilon},
		\end{align*}
		and the lemma follows by letting $\varepsilon$ tend to $0$.
	\end{proof}
	We will also use the below lemma, which relies on elementary algebra and then the Uniqueness of Measures Theorem. This does not use $\iota_p$.
	\begin{lem}\label{lem:Set inversion}
		Let $f:\bZ_p^\times\to \bZ_p^\times$ denote the map $x\mapsto 1/x$.
		Then $f$ maps balls in $\bZ_p^\times$ to balls of the same measure in $\bZ_p^\times$.
		In particular, $f$ preserves $\mu_p$ restricted to $\bZ_p^\times$.
	\end{lem}
	\begin{proof}
		This is a simple calculation.
	\end{proof}
	
\section{Proof of Theorem \ref{thm:spectrum}} \label{section:spectrum Bp}
	To prove Theorem \ref{thm:spectrum}, we will use two propositions.
	The first one puts $\mu_p(\cB^p)$ equal to $\mu_p(\cA^p)$ or $\mu_p(\cC^p)$, depending on a divergence criterion, and splits $\cC^p$ into smaller pieces to be dealt with individually.
	
	In the proof and for the rest of this paper, we will use $\sqcup$ to denote the disjoint union of sets.
	\begin{prop}
		\label{Prop:mup Cp as sum}
		Let $l$ be the minimal $k\in \bN_0$ such that $\psi(n)/n \ge p^{-k}$ for infinitely many $n\in p^k \bN$, if such a $k$ exists. Otherwise, put $l=\infty$ and write $p^{-\infty}=0$. 
		Then
		\begin{align}\label{eq:mup Bp}
			\mu_p(\cB^p)&= \begin{cases}
				\mu_p(\cA^p) = 1	&\text{if } \sum_{p\nmid n} \frac{\phi(n)\psi(n)}{n}=\infty,	\\
				\mu_p(\cC^p)
				&\text{if }\sum_{p\nmid n} \frac{\phi(n)\psi(n)}{n}<\infty,
			\end{cases}
			\\\label{eq:mup Cp}
			\mu_p(\cC^p) &= p^{-l} + \sum_{0\le k<l} \mu_p \big(\cC^p\cap p^k\bZ_p^\times\big).
		\end{align}
		Furthermore, if $k<l$, then $\cC^p\cap p^k\bZ_p^\times= \limsup_{\nu_p(n)=k} \cC_n^p
		$.
	\end{prop}
	\begin{proof}
		We start by proving equation \eqref{eq:mup Bp}.
		If $\sum_{p\nmid n} \frac{\phi(n)\psi(n)}{n}=\infty$, this is \cite[Theorem 2]{Kristensen+Laursen}, so suppose not.
		Then $\sum_{p\nmid n}\mu_p(\cA_n^p)<\infty$.
		If $\sum_{p\mid n}\mu_p(\cA_n^p)=\infty$, then we must have $\psi(n)\ge pn$ infinitely often, and so the definitions imply $\cB_{n}^p\supseteq \cC_n^p \supseteq B_{\bQ_p}(n/1,p) \cap \bZ_p = \bZ_p$ for these $n$ when $n>1$.
		Hence, $\cB^p=\cC^p=\bZ_p$.
		If $\sum_{p\mid n}\mu_p(\cA_n^p)<\infty$, then the Borel-Cantelli Lemma implies $\mu_p(\cA^p)=0$, and the statement follows by equation \eqref{eq:Bp=ApcupCp}.
		
		Moving on to equation \eqref{eq:mup Cp}, notice that
		\begin{align*}
			\cC^p = (\cC^p \cap p^l \bZ_p)\sqcup \bigsqcup_{0\le k<l} \cC^p \cap p^k \bZ_p^\times,
		\end{align*}
		with the convention of $p^\infty = 0$ (as an element of $\bZ_p$) in the case of $l=\infty$.
		Equation \eqref{eq:mup Cp} is then equivalent to $\mu_p(\cC^p \cap p^l \bZ_p) = p^l$.
		If $l=\infty$, this is trivial, so suppose $l<\infty$. 
		Let $\alpha\in p^l\bZ_p$.
		By definition of $l$, there are infinitely many $n\in p^l \bN$ with $\psi(n)\ge p^{-l} n$.
		For these $n$, $|\alpha - n/1 |_p \le p^{-l} \le \psi(n)/n$, and so $\alpha\in\cC^p_n$.
		Hence, $\cC^p\supseteq p^l\bZ_p$, implying $\mu_p(\cC^p \cap p^l \bZ_p) = p^l$ as claimed.
		
		As for the final part of the statement, this is vacuous if $l=0$, so suppose $l>0$ and let $0\le j \le  k<l$.
		Then $\psi(n)/n < p^{-j}$ for all but finitely many $n\in p^j\bN$.
		If $\nu_p(n)=j$, then $\alpha\in\cC_n^p$ would imply 
		\begin{equation*}
			|\alpha|_p = \bigg|\alpha - \frac{n}{a} + \frac{n}{a}\bigg|_p = p^{-j},
		\end{equation*}
		for some $|a|<n$ with $\gcd(a,pn)=1$ when $n$ is sufficiently large, so that
		\begin{align*}
			\limsup_{\nu_p(n)=j} \cC_n^p \subseteq p^j \bZ_p^\times.
		\end{align*}
		If $\nu_p(n)>k$, we would instead find 
		\begin{equation*}
			|\alpha|_p \le \max\bigg\{ \bigg|\alpha - \frac{n}{a}\bigg|_p, \bigg|\frac{n}{a}\bigg|_p\bigg\} < p^{-k},
		\end{equation*}
		for some $|a|<n$ with $\gcd(a,pn)=1$ when $n$ is sufficiently large,
		so that
		\begin{align*}
			\limsup_{\nu_p(n)>k} \cC_n^p \subseteq p^{k+1} \bZ_p.
		\end{align*}
		By the pigeon-hole principle, the proposition is proven upon calculating
		\begin{align*}
			\cC^p\cap p^k \bZ_p^\times &
			= \left(\limsup_{\nu_p(n)>k} \cC_n^p \cap p^k \bZ_p^\times \right)\cup \bigcup_{j=0}^k \limsup_{\nu_p(n)=j} \cC_n^p \cap p^k \bZ_p^\times
			\\&
			= \limsup_{\nu_p(n)=k} \cC_n^p.
		\end{align*}
	\end{proof}
	
	The other proposition leading to Theorem \ref{thm:spectrum} is a bit more involved and may be thought of as a shell-wise zero-full law for $\cC^p$.
	This relates to, and is inspired by, the zero-one laws $|\cA|\in\{0,1\}$ \cite{Gallagher} and $\mu_p(\cA^p)\in\{0,1\}$ \cite{Haynes}.
	\begin{prop}
		\label{Prop:0-1 law}
		Let $k\in\bN_0$.
		Then $\mu_p(\cC^p\cap p^k \bZ_p^\times)\in\{0,(p-1)/p^{k+1}\}$.
	\end{prop}
	The proof of this proposition follows the same overall structure as the proof of the 0-1 law in \cite{Haynes}, with some modifications.
	In that light, it is perhaps not surprising that we will need the below lemma, which corresponds to the less trivial part of \cite[Lemma 2]{Haynes}.
	We will here continue to use $\omega(n)$ to denote the number of prime divisors of $n$. 
	The proof of the lemma will apply the Möbius function $\mu$, which is defined as $\mu(d)=(-1)^{\omega(d)}$ when $d$ is square free, and $\mu(d)=0$ otherwise.
	In this context, the following three facts, which can be found in \cite{Montgomery+Vaughan}, will be applied without proof.
	\begin{align}
		\label{eq:mu sum}
		\sum_{d\mid n}\mu(d) &= \begin{cases}
			1	&\text{if } d=1,	\\
			0	&\text{if } d\ne 1,
		\end{cases}
		\\
		\label{eq:phi as mu sum}
		\sum_{d\mid n}\mu(d)\frac{n}{d} &= \phi(n),
		\\
		\label{eq:phi/n as product}
		\frac{\phi(n)}{n}&= \prod_{q\mid n} (1-q^{-1}).
	\end{align}
	\begin{lem}
		\label{Lemma 2 Haynes}
		If $n>1$ and $\psi(n)>4^{\omega(n)}$, then $\cC_n^p \supseteq p^{\nu_p(n)} \bZ_p^\times$.
	\end{lem}
\begin{proof}
		Put $k=\nu_p(n)$, and let $\alpha\in p^k \bZ_p^\times$.
		Then $n/\alpha\in \bZ_p^\times$, and we write
		\begin{align*}
		\frac{n}{\alpha} = \sum_{m=0}^\infty b_m p^m,
		\qquad
		b_m \in\{0,\ldots, p-1\}\ \forall m\in\bN_0,
		\quad b_0\ne 0.
		\end{align*}
		If $\psi(n)\ge n/p^k$, then clearly $\alpha\in\cC_n^p$ since $|n|_p=|\alpha|_p=p^{-k}$ implies $\big|\alpha - \frac{n}{1}\big|_p \leq p^{-k}$, so suppose not.
		Pick $N\in\bZ$ such that $\psi(n)/n\in[p^{-N},p^{-N+1})$.
		Note that $N>k$.
		Our job of proving $\alpha\in\cC_n^p$ then reduces to finding an $a$ with $|a|< n$ and $\gcd(a,pn)=1$ such that $\nu_p(\alpha - n/a) \geq N$.
		As $\nu_p(\alpha) = k$ and $p\nmid a$, the latter part is equivalent to  $\nu_p(a-n/\alpha)\geq N - k$.
		The proof is then complete if we find an $a$ with $|a|< n$ and $\gcd(a,pn)=1$ such that
		\begin{align*}
		a = \frac{n}{\alpha} + \sum_{m=N-k}^{\infty} c_m p^m = \sum_{m=0}^{N-k-1} b_m p^m + \sum_{m=N-k}^{\infty} (b_m + c_m) p^m,
		\end{align*}
		with $c_m\in\{0,\ldots p-1\}$ for all $m$.
		Write $b = \sum_{m=0}^{N-k-1} b_m p^m$.
		Since $p\nmid b$ as $b_0\ne 0$ and $N>k$, all elements of the set
		\begin{align*}
			A = \{a\in\bZ : |a|< n, \gcd(a,n)=1, a\equiv b\mod p^{N-k}\}
		\end{align*} 
		satisfy these criteria, and so we are done if we can show that $\#A>0$.
		By equation \eqref{eq:mu sum}, we have
		\begin{align*}
		\# A &=
		\sum_{\substack{|a|< n \\ a\equiv b\text{ mod }p^{N-k}	}} \sum_{d\mid \gcd(a,n)} \mu(d)
		= \sum_{\substack{d\mid n \\ p\nmid d}} \sum_{\substack{|a|< n,\ d\mid a \\ a\equiv b\text{ mod }p^{N-k}	}}  \mu(d)
		\\&
		= \sum_{d\mid n /p^k} \mu(d)\sum_{\substack{|l|< n/d \\ l\equiv bd^{-1}\text{ mod } p^{N-k}	}} 1.
		\end{align*}
		To simplify notation, let $\tilde{n}=n/p^k$.
		For each $d\mid \tilde{n}$, pick $x_d\in\bZ$ such that $x_d\equiv bd^{-1}\mod p^{N-k}$, and let $k_d$ denote the difference $\#((-n/d, n/d)\cap (x_d+p^{N-k}\bZ)) - 2n/(p^{N-k}d)$.
		Then
		\begin{align*}
		\# A &= \sum_{d\mid \tilde{n}} \mu(d)\sum_{\substack{|l|< n/d \\ l\in x_d+p^{N-k} \bZ}} 1
		= \sum_{d\mid \tilde{n}} \mu(d) \bigg(\frac{2n}{p^{N-k} d}+k_d\bigg)
		\\&
		= \frac{2p^k}{p^{N-k}}\sum_{d\mid \tilde{n}} \mu(d) \frac{\tilde{n}}{d} + \sum_{d\mid \tilde{n}} \mu(d)k_d
		=2 p^{-N} p^{2k}\phi(\tilde{n}) + \sum_{d\mid \tilde{n}} \mu(d)k_d
		\\&
		\geq 2p^k\frac{\psi(n)}{n}\phi(n) + \sum_{d\mid \tilde{n}} \mu(d)k_d,
		\end{align*}
		where the final equality and the inequality follow from equation \eqref{eq:phi as mu sum} and the choices of $N$ and $k$, respectively.
		Notice that $|k_d|\leq 1$, which combined with the assumption that $\psi(n)>4^{\omega(n)}$ implies
		\begin{align*}
		\# A & > 2p^k \frac{4^{\omega(n)}}{n} \phi(n) - \sum_{d\mid \tilde{n}} |\mu(d)|
		= 2p^k\frac{\phi(n)}{n} 4^{\omega(n)} - 2^{\omega(\tilde{n})}.
		\end{align*}
		By equation \eqref{eq:phi/n as product}, we have 
		\begin{align*}
		2\frac{\phi(n)}{n} 4^{\omega(n)}
		= 2\prod_{\substack{q\mid n \\ q\text{ prime}	}} 4 (1-q^{-1}) \geq 2\prod_{\substack{q\mid n \\ q\text{ prime}}} 2 = 2^{\omega(n)+1},
		\end{align*}
		and so $\#A > 2^{\omega(n)}>0$.
		We conclude $\alpha\in\cC_n^p$ and thus $\cC_n^p \supseteq p^k\bZ_p^\times$, and the proof is complete.
	\end{proof}
We will also need the below lemma, which is essentially a specialisation of Corollary \ref{Cor Gallagher} in the context of $\cC^p$.
\begin{lem}\label{lem:Gallagher Cp}
	Let $k\in\bN_0$, and
	suppose $\supp(\psi)\subseteq p^k \bN\setminus p^{k+1}\bN$.
	Then $\mu_p\big(\cC^p(x\psi)\cap p^k \bZ_p^\times\big) = \mu_p\big(\cC^p( \psi)\cap p^k \bZ_p^\times\big)$ for all $x>0$.
\end{lem}
\begin{proof}
	There is nothing to prove if $x=1$ or if $\psi(n)>0$ for only finitely many $n$, so suppose neither is the case.
	By replacing $\psi$ with $\psi'=x\psi$ and $x$ by $x'=1/x$ if necessary, we may assume without loss of generality that $x\in(0,1)$.
	If $x\psi(n)\ge 4^{\omega(n)}$ infinitely often, then Lemma \ref{Lemma 2 Haynes} implies that $\cC^p(\psi),\cC^p(x\psi)\supseteq p^k \bZ_p^\times$, so suppose not.
	Then, for $n$ sufficiently large,
	\begin{align*}
		\frac{\psi(n)}{n} \le x^{-1} \frac{4^{\omega(n)}}{n} 
		\le 16 x^{-1} \frac{4^{\log_5(n)}}{n} 
		\underset{n\to\infty}{\longrightarrow} 0.
	\end{align*}
	Since $\mu_p(\cC^p_n(\psi))=0$ when $\psi(n)=0$, we have by the Borel--Cantelli Lemma that
	\begin{align*}
		\mu_p\big(\cC^p(\psi)\big) &= \mu_p\Big(\limsup_{\psi(n)>0}\cC^p_n(\psi)\Big),
		\\
		\mu_p\big(\cC^p(x\psi)\big) &= \mu_p\Big(\limsup_{\psi(n)>0}\cC^p_n(x\psi)\Big).
	\end{align*}
	Notice that for $\psi(n)>0$, $\cC^p_n(\psi)$ is a finite union of proper balls $B_{i_n},\ldots, B_{i_{n+1}-1}$ of radius $\psi(n)/n \underset{n\to\infty}{\longrightarrow} 0$, and that each ball $B_{i_n+j}$ is matched one-to-one by a ball $U_{i_n+j}$ from $\cC^p_n(x\psi)$ satisfying
	\begin{align*}
		U_{i_n+j} \subseteq B_{i_n+j},
		\quad
		\mu_p(U_{i_n+j})\ge \frac{x}{p} \mu_p(B_{i_n+j}).
	\end{align*}
	From Corollary \ref{Cor Gallagher}, it follows that $\mu_p(\cC^p(x\psi)) = \mu_p(\cC^p(\psi))$.
	This completes the proof since clearly $\cC^p(x\psi)\subseteq \cC^p(\psi)$.
	\end{proof}
	
	We are now ready to complete the proof of the shell-wise $p$-adic zero-full law.
	The remaining part of the proof is where it differs the most from the proof of the 0-1 law in \cite{Haynes}, though it still follows the same overall idea.
	\begin{proof}[Proof of Proposition \ref{Prop:0-1 law}]
		Let $l$ be defined as in Proposition \ref{Prop:mup Cp as sum}.
		If $k\ge l$, then $l<\infty$, and it follows from the proof of that proposition that $\cC^p(\psi)\supseteq p^l\bZ_p\supseteq p^k\bZ_p^\times$, and we are done, so suppose $k>l$.
		We then have that $\cC^p(\psi)\cap p^k \bZ_p^\times = \limsup_{\nu_p(n)=k}\cC^p_n(\psi)$, and we may hence assume, without loss of generality, that $\supp(\psi)\subseteq p^k\bN\setminus p^{k+1}\bN$.
		Following the arguments in the proof of Lemma \ref{lem:Gallagher Cp}, $\limsup_{n\to\infty} \psi(n)/n > 0$ would likewise imply $\cC_n^p(\psi)\supseteq p^k \bZ_p^\times$, so suppose $\limsup_{n\to\infty} \psi(n)/n =0$.
		We then have, for any fixed $j$, that
		\begin{align}
		\label{eq:pj psi < n}
		\frac{\psi(n)}{n} < p^{-j}
		\quad\text{ for all } n\ge N_j, \text{ for some } N_j\in\bN.
		\end{align}
		Let $\tau_1:\bZ_p^\times\to\bZ_p^\times$ and $\tau_2:p^k\bZ_p^\times\to p^k\bZ_p^\times$ be given by
		\begin{align*}
		\tau_1 (b) &=\begin{cases}
		\sum_{m=0}^{\infty}b_{m+1}p^m	&\text{if }b_1 \ne 0,	\\
		1 + \sum_{m=0}^{\infty}b_{m+1}p^m	&\text{if }b_1 = 0,
		\end{cases}
		\\
		\tau_2 (p^k b) &= p^k/\tau_1 (b),
		\end{align*}
		for $b=\sum_{m=0}^{\infty}b_m p^m \in \bZ_p^\times$.
		For $K\ge 2$ and $b=\sum_{m=0}^{K-1} b_m p^m\in\bZ_p^\times$, note that
		\begin{align}
			\label{eq:taup(B) a ball}
			\tau_1\big(b + p^{K} \bZ_p\big)
			= \tau_1(b_1) + \sum_{m=1}^{K-2}b_{m+1} p^m + p^{K-1} \bZ_p.
		\end{align}
		Thus, $\tau_1$ maps balls of centre $b$ and radius $p^{-K}$ to balls of centre $\tau_1(b)$ and radius $p^{1-K}$ when $K\ge 2$.
		This makes any restriction of $\tau_1$ to a ball in $\bZ_p^\times$ of radius at most $p^{-2}$ into a homeomorphism onto its image as it is clearly a bijection under such a restriction.
		By Lemma \ref{lem:Set inversion}, these properties extend to a restriction of $\tau_2$ when replacing $K$ by $M\ge k+2$.
		Let $B$ be a ball of radius at most $p^{-k-2}$.
		The inverse of $\tau_2$ restricted to $B$, $\tau_2|_B^{-1}$ is thus measurable and has a push-forward measure, 	$(\tau_{2}|_B^{-1})_{\#\mu}$,  satisfying
		\begin{align*}
			(\tau_{2}|_B^{-1})_{\#\mu}(\tilde{B}) = \mu_p(\tau_2(\tilde{B})) = p\mu_p(\tilde{B}),
		\end{align*}
		for all balls $\tilde{B}\subseteq \tau_2(B)$.
		By the proof of Lemma \ref{Lemma lambda and mup}, this means that $(\tau_{2}|_B^{-1})_{\#\mu} = p\mu_p$, i.e.,
		\begin{align}
			\label{eq:mup taup = p mup}
			\mu_p(\tau_2(A))=p\mu_p(A),
		\end{align}
		for all Borel subsets $A\subseteq B$.
				
		Let $\alpha \in\cC^p(\psi)\cap p^k\bZ_p^\times$ and write
		\begin{align*}
		\alpha' := \alpha/p^k = \sum_{m=0}^{\infty} a_m p^m \in \bZ_p^\times,
		\qquad
		a_i\in\{0,\ldots, p-1\}.
		\end{align*}
		Let $n\in\bN$ such that $\alpha \in\cC^p_n (\psi)$.
		Since $\lim_{n\to\infty}\psi(n)/n = 0$, we have $\psi(n)/n \le p^{-k-2}$ for $n$ large enough. 
		By the assumption on the support of $\psi$, we may write $n = p^k n'$ for some $n'\in\bN\setminus p\bN$ and pick some $|a|< |n|$ with $\gcd(a,pn)=1$ such that
		\begin{align*}
		|n'\alpha' - a|_p = \bigg|\frac{n}{\alpha} - a\bigg|_p 
		= p^k \bigg| \alpha - \frac{n}{a}\bigg|_p \leq p^k \frac{\psi(n)}{n}.
		\end{align*}
		If $a_1\ne 0$, put $a'=(a-a_0 n')/p$.
		Then
		\begin{align*}
		\bigg|	\frac{n}{\tau_2(\alpha)} - a'	\bigg|_p &=
		\bigg|	n'\tau_1(\alpha') - \frac{a-a_0 n'}{p}	\bigg|_p 
		= \bigg|	\frac{n'(\alpha'-a_0)-(a-a_0n')}{p}	\bigg|_p
		\\&
		= p|n'\alpha' - a|_p 
		\leq p^{k+1}\frac{\psi(n)}{n},
		\\
		|a'| &= \bigg|\frac{a-a_0 n'}{p}\bigg| \le \frac{|a| + (p-1) n}{p} < n,
		\\
		a' &= \frac{a-a_0 n'}{p} = \frac{a-(a_0+p a_1) n'}{p} + a_1 n'.
		\end{align*}
		Since $p^2\mid a-\alpha'n'$, we have $p\mid (a-(a_0+p a_1) n')/p\in\bZ$.
		From the last equation, we can thus deduce that $\gcd(a',pn')=\gcd(a',pn)=1$, as $p\nmid a_1 n'$ and $\gcd(a,n') =1$.
		Combined with the other two equations, we conclude that $\tau_2(\alpha)\in\cC^p_n(p\psi)$.
		If $a_1 = 0$, we put $a'=(a+(p-a_0) n')/p$ and reach the same conclusion, based on similar calculations.
		Hence, $\tau_2(\cC^p(\psi)) \subseteq \cC^p(p\psi)$.
		By induction, we then have
		\begin{align}
		\label{eq tau 2j in union Cp(pipsi)}
		\tau_2^j(\cC^p(\psi)) \subseteq \cC^p(p^j\psi),
		\end{align}
		for all $j\in\bN$, where $\tau_2^j$ denotes the composition of $j$ copies of $\tau_2$.
		
		If $\mu_p(\cC^p(\psi))=0$, we are done, so suppose that $\mu_p(\cC^p (\psi))>0$.
		Then Lemma \ref{Lemma padic density theorem} implies that $\cC^p(\psi)$ contains a point $\alpha$ of density 1, i.e.,
		\begin{align*}
		\frac{\mu_p(\cC^p(\psi)\cap B_{M}(\alpha))}{\mu_p(B_M(\alpha))} \underset{M\to\infty}{\longrightarrow} 1,
		\end{align*}
		where we use $B_M(\alpha)$ as short-hand notation of the ball $B_{\bQ_p}(\alpha,p^{-M}) = \alpha + p^M\bZ_p$.
		Let $\varepsilon>0$ and pick $M\ge k+2$ such that
		\begin{align*}
		\mu_p(\cC^p(\psi)\cap B_M(\alpha))\geq (1-\varepsilon) \mu_p(B_M(\alpha)) = (1-\varepsilon) p^{-M}.
		\end{align*}
		Pick $x = p^k\sum_{m=0}^{M-k-1} x_m p^m\in\bN$ such that $B_M(x)=B_M(\alpha)$.
		Put $A = \tau_2^{M-k-1} (\cC^p\cap B_M(x))$.
		By Lemma \ref{lem:Gallagher Cp}, equation \eqref{eq:taup(B) a ball}, and inclusion \eqref{eq tau 2j in union Cp(pipsi)}, we have
		\begin{align}
			\nonumber
			\mu_p\big(\cC^p(\psi)\cap p^k \bZ_p^\times\big) &= \mu_p\big(\cC^p(p^{M-k}\psi)\cap p^k \bZ_p^\times\big) 
			\\& \label{eq:Cp to tauM-k Cp}
			\geq \mu_p\big(\tau_2^{M-k} (\cC^p(\psi)\cap B_M(x))\big) 
			= \mu_p(\tau_2 (A)).
		\end{align}
		Let $\tilde{x}=\tau_2^{M-k-1}(x)$ and note that $\tilde{x} = p^k\tilde{x}_0$ for some $\tilde{x}_0\in\{1,\ldots, p-1\}$.
		Then
		\begin{align*}
		A \subseteq B_{k+1}(\tilde{x}) = p^k \bigsqcup_{i=0}^{p-1}\big( \tilde{x}_0 + ip + p^2 \bZ_p\big).
		\end{align*}
		By applying equation \eqref{eq:mup taup = p mup} and then later the above inclusion, we find
		\begin{align*}
		\mu_p (\tau_2(A)) &\ge \sum_{i=1}^{p-1} \mu_p(\tau_2(A\cap B_{k+2}(\tilde{x}+ip))) 
		\\&
		= p \sum_{i=1}^{p-1} \mu_p(A\cap B_{k+2}(\tilde{x}+ip))
		\\&
		= p \mu_p(A\setminus B_{k+2}(\tilde{x}))
		\ge p (\mu_p(A) - p^{-k-2}) 
		\\&
		\ge (p-1)\mu_p(A)
		\end{align*}
		Meanwhile, an iterative application of equations \eqref{eq:taup(B) a ball} and \eqref{eq:mup taup = p mup} yield
		\begin{align*}
		\mu_p(A) &= p^{M-k-1} \mu_p(\cC^p\cap B_M(x))
		\geq \frac{1 - \varepsilon}{p^{k+1}},
		\end{align*}
		so that $\mu_p\big(\cC^p(\psi)\cap p^k \bZ_p^\times\big) \ge \frac{p-1}{p^{k+1}}(1-\varepsilon)$, by inequality \eqref{eq:Cp to tauM-k Cp}.
		This completes the proof.
	\end{proof}
	\begin{proof}[Proof of Theorem \ref{thm:spectrum}]
		We start by the `only if' parts.
		For $\cC^p$, it follows by Proposition \ref{Prop:0-1 law} upon noting
		\begin{align*}
			\cC^p &= (\cC^p\cap\{0\})\sqcup \bigsqcup_{k=0}^\infty \big(\cC^p\cap p^k \bZ_p^\times\big).
		\end{align*}
		As for $\cB^p$, suppose $\mu_p(\cB^p)<1$.
		By Proposition \ref{Prop:mup Cp as sum}, $\sum_{p\nmid n} \phi(n)\psi(n)/n <\infty$, $\mu_p(\cB^p)=\mu_p(\cC^p)$, and $\cC^p\cap \bZ_p^\times = \limsup_{p\nmid n} \cC^p_n$.
		Since
		\begin{align*}
			\sum_{p\nmid n}	\mu_p(\cC^p_n) \leq \sum_{p\nmid n} \frac{2\phi(n)\psi(n)}{n} < \infty,
		\end{align*}
		the Borel-Cantelli Lemma  implies $\mu_p(\cC^p\cap\bZ_p^\times)=0$, and we are done by the above consideration of $\cC^p$.

		The `if' parts of the theorem are already dealt with in the proof of Theorem 3 of \cite{Kristensen+Laursen} (at least for $\cB^p$), but we will repeat the argument here for clarity, shortened by means of Proposition \ref{Prop:mup Cp as sum}.
		Let $x_k\in\{0,1\}$ for $k\in\bN_0$ and define $\psi:\bN\to\bR_{\geq 0}$ by
		\begin{align*}
			\psi(n) &= x_{\nu_p(n)}\frac{n}{p^{\nu_p(n)+1}}.
		\end{align*}
		Let $k\in\bN_0$, and let $q> p$ be some large prime.
		If $x_k=0$, we clearly have $\limsup_{\nu_p(n)=k}\cC^p_n =\emptyset$. 
		If, on the other hand, $x_k = 1$, then $\psi(p^kq)/(p^kq)=p^{-k-1}$, and so
		\begin{align*}
			\cC^p_{p^kq} &= \bigcup_{\substack{|a|< p^k q \\ \gcd(a,p^{k+1}q)=1 }} B_{\bQ_p}\bigg(\frac{p^k q}{a}, p^{-k-1}\bigg)
			\\&
			\supseteq \bigcup_{a=1}^{p-1} p^k \frac{q}{a} + p^{k+1}\bZ_p
			= p^k \bigcup_{a=1}^{p-1}	\frac{q}{a} + p\bZ_p	
		\end{align*}
		Since $q$ is a unit modulo $p$, and since inversion and multiplication by units only permute the set of units modulo $p$, this implies that $\cC^p_{p^kq}\supseteq p^k\bZ_p^\times$.
		As there are infinitely many such $q$, we get $\cC^p\supseteq p^k\bZ_p^\times$.
		Note that if $x_0=1$, then $\sum_{p\nmid n} \phi(n)\psi(n)/n = \infty$.
		By Proposition \ref{Prop:mup Cp as sum}, this means that
		\begin{align*}
			\mu_p(\cC^p) &= \sum_{x_k = 1} \mu_p(p^k\bZ_p^\times) 
			= \sum_{k=0}^\infty x_k (p-1)p^{-k-1},
			\\
			\mu_p(\cB^p) &= \begin{cases}
				1	&\text{if } x_0 = 1,	\\
				\sum_{k=1}^\infty x_k (p-1)p^{-k-1}	&\text{if } x_0 = 0,
			\end{cases}
		\end{align*}
		as the value $l$ from the proposition is clearly infinite.
		This completes the proof of	Theorem \ref{thm:spectrum}.
	\end{proof}
	
\section{Proof of Theorem \ref{thm:spectrum var} for $p<\infty$}\label{section:mult Ap} \setcounter{case}{0}
	The proof of Theorem \ref{thm:spectrum var} follows the same main idea as the `if' part of Theorem \ref{thm:spectrum}, though some details are different.
	The main difference is that additional care is needed for the choice of the support.
	This will rely heavily on the below theorem due to Dirichlet, which may be found in \cite{Montgomery+Vaughan}.
	\begin{thmDir}
		Let $a,b\in \bN$ such that $\gcd(a,b)=1$. Then there are infinitely many primes $q\equiv a\mod b$. 
	\end{thmDir}
	
	To simplify the notation, the symbol $\pm$ will be used to implicitly denote the union of the cases of $+$ and $-$, respectively, in place of $\pm$.
	We thus write
	\begin{align*}
	B(\pm a^{\pm 1}, r) &= B(a^{\pm 1}, r) \cup B(- a^{\pm 1}, r) \\&
	=  B(a, r) \cup B(a^{- 1}, r) \cup B(-a, r) \cup B(-a^{- 1}, r).
	\end{align*}
	Furthermore, we continue using $\sqcup$ to denote the disjoint union of sets.
	Recall that we already proved Theorem \ref{thm:spectrum var} for $p=\infty$ in Section \ref{section:main}.
	\begin{proof}[Proof of Theorem \ref{thm:spectrum var} for $p<\infty$]
		Let $x \in[0,1]$.
		If $x=1$, the statement is trivial, so suppose not.
		We then write $x= \sum_{k=0}^\infty x_k p^{-k-1}$ where the $x_k$ are chosen such that $\liminf_{k\to\infty} x_k<p-1$.
		Pick $K = \min\{k\in\bN_0: x_k<p-1\}$, and let $g\in\bN$ be such that $g+p\bZ_p$ generates the multiplicative group $\bZ_p/p\bZ_p$.
		As the rest of the construction will depend on the prime $p$, we split into four cases.
		In the first two cases, which construct $\psi$ for primes $p>5$ according to their congruency classes modulo 4, we do not put any further restrictions on the choice of $g$.
		In the other two cases, which deal with $p=2$ and $p=3,5$, respectively, we will need some further restrictions on $g$ and therefore fix a specific value, depending on $p$.
		
		\case{$5<p\equiv 1\mod 4$} \label{case:5<p equiv 1 mod p}
		For $a\in\{0,\ldots,p-1\}$ and $k\in\bN_0$, define
		\begin{align*}
		I_a &:= \begin{cases}
			\emptyset	&\text{if } a<4,	\\
			\{1,g^{\frac{p-1}{4}}\}\cup\{g^i : 2\le i\le a/4 \}	&\text{if } 4\le a< p-1,	\\
			\{1,\ldots, p-1\}	&\text{if } a=p-1,
		\end{cases}
		\linebreakLong
		r_k &:= \begin{cases}
			x_k - 4\left\lfloor \frac{x_k}{4} \right\rfloor + \sum_{l=k+1}^\infty x_l p^{k-l}	&\text{if } x_k<p-1,	\\
			0	&\text{if } x_k=p-1.
		\end{cases}
		\end{align*}
		Since clearly $r_k\in[0,4]$, we may write
		\begin{align*}
		\frac{r_k}{4} = \sum_{i=1}^\infty b_{k,i} p ^{-i},
		\qquad b_{k,i}\in\{0,1,\ldots, p-1\}.
		\end{align*}
		Based on this, we construct $\psi$ as
		\begin{align*}
		\psi(n) := \begin{cases}
			f_k(q)	&\text{if } n=p^k q, \text{ where }
			k\le K,	\\
			0	&\text{otherwise},
		\end{cases}
		\end{align*}
		where we use $q$ to denote primes other than $p$, and
		\begin{align*}
		f_k (q) &= \begin{cases}
			q/p	&\text{if } q\equiv m\text{ mod }p, \text{ where }
			m\in I_{x_k},	\\[0.1 cm]
			q/p^{i+1}	&\begin{aligned}
				\text{if }& q\equiv g+b'p^i\text{ mod } p^{i+1},
				\\[0.1 cm]&
				\text{where }
				1\le b'\le b_{k,i}
				,\ i\in\bN,
			\end{aligned}	\\
			0	&\text{otherwise}.
		\end{cases}
		\end{align*}
	
		Let $k\le K$.
		Then
		\begin{align}\nonumber
		\fA_{p^k q}^p &= \begin{cases}
			B_{\bQ_p}(\pm q^{\pm 1} p^k, p^{-k-1})	&\text{if } q\equiv m\text{ mod }p, \text{ where }
			m\in I_{x_k},	\\[0.1 cm]
			B_{\bQ_p}(\pm q^{\pm 1}p^k, p^{-k-i-1})	&\begin{aligned}
				\text{if }& q\equiv g+b'p^i\text{ mod } p^{i+1},
				\\&
				\text{where } 1\le b'\le b_{k,i},\ i\in\bN,
			\end{aligned}
			\\[0.1 cm]
			B_{\bQ_p}(\pm q^{\pm 1}p^k, 0)	&\text{otherwise}
		\end{cases}
		\linebreakLong& \label{eq:Apkq}
		= p^k\begin{cases}
			\pm q^{\pm 1} + p\bZ_p	&\text{if } q\equiv g^i \text{ mod }p, \text{ where } g^i\in I_{x_k},
			\\[0.1 cm]
			\pm q^{\pm 1} + p^{i+1}\bZ_p	&\begin{aligned}
			\text{if }& q\equiv g+b'p^i\text{ mod } p^{i+1},
			\\&
			\text{where } 1\le b'\le b_{k,i},\ i\in\bN,
			\end{aligned}	\\[0.1 cm]
			\{\pm q^{\pm 1}\}	&\text{otherwise}
		\end{cases}
		\linebreakLong&\nonumber
		\subseteq p^k\bZ_p^\times.
		\end{align}
		If $k<K$, then $I_{x_k} = \{1,\ldots, p-1\}$, and so we are in the first case for all $q$, implying that $\limsup_{\substack{q\to\infty}} \fA^p_{p^k q}
		= p^k\bZ_p^\times$, by \DirThm.
		By Dirichlet's pigeonhole principle, this means that
		\begin{equation}\label{eq:Apn,k<K}
			\fK^p\supseteq \fA^p\supseteq \limsup_{\substack{n\to\infty \\ n=p^k q,\ k< K}} \fA^p_{n}
			= \bigcup_{k=0}^{K-1} \limsup_{\substack{q\to\infty}} \fA^p_{p^k q}
			= \bZ_p \setminus p^{K}\bZ_p.
		\end{equation}
		We are thus left to consider $\fK^p\cap p^{K}\bZ_p$ and $\fA^p\cap p^{K}\bZ_p$.
		Since $q$ is prime, and $\psi(p^kq)/(p^kq) \le p^{-k-1}$,
		we note, for arbitrary $k\in\bN_0$,
		\begin{align}
			\label{eq:Kpk0q}
			\fK^p_{p^kq} &\subseteq
			\fA^p_{p^k q}\cup \bigcup_{0< j\le k/2} p^{k-2j}\bZ_p^\times.
		\end{align}
		Since $\fK^p_n$ is finite when $\psi(n)=0$, and $\fK^p_{qp^k}\subseteq \bZ_p\setminus p^{k+1} \bZ_p$ by equation \eqref{eq:Apkq} and inclusion \eqref{eq:Kpk0q}, we have
		\begin{align*}
			\mu_p(\fK^p\cap p^{K}\bZ_p) &= \mu_p\left( \limsup_{n\to\infty} \fK^p_n \cap p^{K}\bZ_p \right) 
			= \mu_p\left( \limsup_{q\to\infty} \fK^p_{p^{K} q} \cap p^{K}\bZ_p \right).
		\end{align*}
		Applying inclusion \eqref{eq:Kpk0q} and then equation \eqref{eq:Apkq} once more, we find
		\begin{align*}
			\mu_p(\fK^p\cap p^{K}\bZ_p) &= \mu_p\left( \limsup_{q\to\infty} \fA^p_{p^{K} q} \cap p^{K}\bZ_p \right)
			\\&
			= \mu_p\left( \limsup_{q\to\infty} \fA^p_{p^{K} q}\bZ_p \right)
			= \mu_p(\fA^p\cap p^{K}\bZ_p)
			.
		\end{align*}
		Since $x_k = p-1$ for $0\le k<K$, inclusion \eqref{eq:Apn,k<K} implies
		\begin{align*}
			\mu_p(\fK^p) = \mu_p(\fA^p) &= \mu_p(\bZ_p \setminus p^{K}\bZ_p) + \mu_p\left( \limsup_{q\to\infty} \fA^p_{p^{K} q} \right)
			\\&
			= \sum_{k=0}^{K-1}x_k p^{-k-1} + \mu_p\left( \limsup_{q\to\infty} \fA^p_{p^{K} q} \right).
		\end{align*}
		The theorem hence follows for the current case if we can show that
		\begin{equation}
		\label{eq:mup(A cap pk0Zp)}
			\mu_p\left( \limsup_{q\to\infty} \fA^p_{p^{K} q} \right) 
			= \sum_{k=K}^\infty x_k p^{-k-1}.
		\end{equation}
		Applying \DirThm\ on equation \eqref{eq:Apkq} for $k=K$, we find
		\begin{align}\nonumber
			\limsup_{q\to\infty} \fA_{p^{K} q}^p	
			= &\bigcup_{m\in I_{x_{K}}} (\pm m^{\pm 1}p^{K} + p^{K+1}\bZ_p) 
			\\&\ \ \nonumber
			\cup \bigcup_{i=1}^\infty \bigcup_{b'=1}^{b_{{K},i}} (\pm (g+b'p^i)^{\pm 1}p^{K} + p^{{K}+i+1}\bZ_p)
			\linebreakLong	\label{eq:limsupApkq disjoint}
			& \begin{aligned}
				=&\bigsqcup_{m\in I_{x_{K}}} (\pm m^{\pm 1}p^{K} + p^{K+1}\bZ_p) 
				\\&
				\sqcup \bigcup_{i=1}^\infty \bigcup_{b'=1}^{b_{{K},i}} (\pm (g+b' p^i)^{\pm 1}p^{K} + p^{{K}+i+1}\bZ_p),
			\end{aligned}
		\end{align}
		using that $(\pm m_1^{\pm 1}p^{K} + p^{{K}+1}\bZ_p)\cap (\pm m_2^{\pm 1}p^{K} + p^{{K}+1}\bZ_p)=\emptyset$ for distinct $m_1,m_2\in\{g^j : 0\le j\le (p-1)/4\}$.
		We clearly also have 
		\begin{align*}
		(\pm (g+b'_1 p^i)^{\pm 1}p^{K} + p^{K+i+1}\bZ_p) \cap (\pm (g+b'_2p^j)^{\pm 1}p^{K} + p^{K+j+1}\bZ_p) = \emptyset,
		\end{align*}
		for $i\ne j$ or $b_1'\ne b_2'$.
		Applying this to the above calculation, we find
		\begin{align}\nonumber
			\mu_p\left(\limsup_{q\to\infty} \fA_{p^{K} q}^p\right) &=
			\sum_{m\in I_{x_{K}}}\mu_p(\pm m^{\pm 1}p^{K} + p^{K+1}\bZ_p) 
			\\&\nonumber	\quad
			+ \sum_{l=1}^\infty \sum_{b'=1}^{b_{K,l}}\mu_p(\pm (g+b'p^l)^{\pm 1}p^{K} + p^{K+l+1}\bZ_p)
			\linebreakLong& \label{eq:mup(limsup Apkq)}
			\begin{aligned}
			=p^{-K}\Bigg(&\sum_{m\in I_{x_{K}}} \mu_p(\pm m^{\pm 1} + p\bZ_p) 
			\\&
			+ \sum_{l=1}^\infty   \sum_{b'=1}^{b_{K,l}} \mu_p (\pm (g+b'p^l)^{\pm 1} + p^{l+1}\bZ_p)\Bigg).
		\end{aligned}
		\end{align}
		In order to determine the first sum in equation \eqref{eq:mup(limsup Apkq)}, note that
		\begin{equation*}
			-g^i \equiv g^{\frac{p-1}{2} + i},
			\quad
			g^{-i}\equiv g^{p-1-i},
			\quad
			-g^{-i} \equiv g^{\frac{p-1}{2} - i}
		\end{equation*}
		modulo $p$.
		For $1\le i < (p-1)/4$, we have
		\begin{equation*}
			\begin{aligned}
			\frac{p-1}{4}&< \frac{p-1}{2} - i 
			<\frac{p-1}{2}
			<\frac{p-1}{2} + i
			<3 \frac{p-1}{4}
			\linebreakLong&
			<p-1 - i
			<p-1,
			\end{aligned}
		\end{equation*}
		implying that $g^{i}, g^{-i}, -g, -g^{-i}$ are not congruent modulo $p^l$ for $l\ge 1$.
		A similar argument yields $1= 1^{-1}\not\equiv -1 = -1^{-1}$ and 
		$$g^{(p-1)/4} \equiv -g^{-(p-1)/4}\not\equiv - g^{(p-1)/4} \equiv g^{-(p-1)/4}$$ 
		modulo $p$.
		Hence, for integers $1\le i < (p-1)/4$, $0\le b'\le p-1$, $l\ge 1$, and $j\in\{0,(p-1)/4\}$,
		\begin{equation}\label{eq:mup(pm g+b'pi)}
			\begin{aligned}
				\mu_p (\pm (g^i+b'p^l)^{\pm 1} + p^{l+1}\bZ_p) &= 4/p^{l+1},	\linebreakLong
				\mu_p (\pm g^{\pm j} + p\bZ_p) &= 2/p.
			\end{aligned}
		\end{equation}
		We are now ready to handle the first sum of \eqref{eq:mup(limsup Apkq)}.
		For $x_{K}\ge 4$, we find
		\begin{align*}
			\sum_{m\in I_{x_{K}}}\mu_p(\pm m^{\pm 1} + p\bZ_p)
			= &\sum_{i=2}^{\left\lfloor x_{K} /4\right\rfloor} \mu_p(\pm g^{\pm(p-1)/4}+p\bZ_p)
			\\&
			+ \mu_p(\pm 1^{\pm 1}+p\bZ_p) + \mu_p(\pm g^{\pm(p-1)/4}+p\bZ_p) 
			\linebreakLong
			=& \left(\sum_{i=2}^{\left\lfloor x_{K}/4 \right\rfloor} \frac{4}{p}\right) + \frac{2}{p} + \frac{2}{p}
			= \frac{4}{p} \left\lfloor  \frac{x_k}{4} \right\rfloor.
		\end{align*}
		For $x_{K}<4$, $I_{x_{K}}=\emptyset$, and so we reach the same conclusion in that case.
		
		Equation \eqref{eq:mup(limsup Apkq)} and the definitions of $b_{a,i}$, $r_a$, and $x_k$ then allow us to conclude equation \eqref{eq:mup(A cap pk0Zp)} as we see
		\begin{align*}
		\mu_p(\limsup_{q\to\infty} \fA_{p^K q}^p) &= p^{-K}
		\left(	\frac{4}{p} \left\lfloor\frac{x_{K}}{4}\right\rfloor + \sum_{i=1}^\infty   \sum_{b'=1}^{b_{K,i}} \frac{4}{p^{i+1}}	\right)
		\linebreakLong&
		=  p^{-K-1}4
		\left(\left\lfloor\frac{x_{K}}{4}\right\rfloor + \sum_{i=1}^\infty b_{k,i} p^{-i}\right)
		\linebreakLong&
		= p^{-K-1}4
		\left(\left\lfloor\frac{x_{K}}{4}\right\rfloor + \frac{r_{K}}{4}\right)
		\linebreakLong&
		= p^{-K-1}\sum_{l=K}^\infty p^{K-l}x_l
		= \sum_{l=K}^\infty x_l p^{-l-1}.
		\end{align*}
		
		\case{$5<p\equiv 3\mod 4$}\label{case:5<p equiv 3 mod p}
		This case closely follows the structure of case \ref{case:5<p equiv 1 mod p}.
		The main difference is a modification to the definitions of $I_a$ and  $r_k$, so that
		\begin{align*}
			I_a &:= \begin{cases}
				\emptyset	&\text{if } a<2,	\\
				\{1\}\cup\{g^i : 2\le i\le (a+2)/4 \}	&\text{if } 2\le a< p-1,	\\
				\{1,\ldots, p-1\}	&\text{if } a=p-1,
			\end{cases}
			\linebreakLong
			r_k &:= \begin{cases}
			\sum_{l=k}^\infty x_l p^{k-l}	&\text{if } x_k<2,	\\
			\sum_{l=k}^\infty x_l p^{k-l} - 4\left\lfloor \frac{x_k-2}{4} \right\rfloor - 2	&\text{if } 2\le x_k<p-1,	\\
			0	&\text{if } x_k=p-1.
		\end{cases}
		\end{align*}
		Note that all arguments until and including equation \eqref{eq:mup(pm g+b'pi)} remain valid, except that we no longer have an integer $j=(p-1)/4$.
		For $x_{K}<2$, the remaining arguments are unchanged, so suppose $x_{K}\ge 2$.
		Then
		\begin{align*}
		\sum_{m\in I_{x_{K}}}\mu_p(\pm m^{\pm 1} + p\bZ_p)
		= &\sum_{i=2}^{\left\lfloor (x_{K}+2) /4\right\rfloor} \mu_p(\pm g^{\pm(p-1)/4}+p\bZ_p)
		\\&
		+ \mu_p(\pm 1^{\pm 1}+p\bZ_p)
		\linebreakLong
		=& \left(\sum_{i=2}^{\left\lfloor (x_{K}+2)/4 \right\rfloor} \frac{4}{p}\right) + \frac{2}{p}
		= \frac{4}{p} \left\lfloor  \frac{x_k-2}{4} \right\rfloor + \frac{2}{p}.
		\end{align*}
		Applying this and equation \eqref{eq:mup(pm g+b'pi)} to equation \eqref{eq:mup(A cap pk0Zp)}, we conclude
		\begin{align*}
		\mu_p(\limsup_{q\to\infty} \fA_{p^{K} q}^p) &= p^{-K}\left(	\frac{4}{p} \left\lfloor  \frac{x_k-2}{4} \right\rfloor + \frac{2}{p} + \sum_{i=1}^\infty   \sum_{b'=1}^{b_{K,i}} \frac{4}{p^{i+1}}	\right)
		\linebreakLong&
		=  p^{-K-1}4
		\left(\left\lfloor  \frac{x_k-2}{4} \right\rfloor + \frac{1}{2} + \sum_{i=1}^\infty b_{k,i} p^{-i}\right)
		\linebreakLong&
		= p^{-K-1}4
		\left(\left\lfloor  \frac{x_k-2}{4} \right\rfloor + \frac{1}{2} + \frac{r_{K}}{4}\right)
		\linebreakLong&
		= p^{-K-1}\sum_{l=K}^\infty p^{K-l}x_l
		= \sum_{l=K}^\infty x_l p^{-l-1}.
		\end{align*}
		
		\case{$p=2$}\label{case:p=2}
		In this case, we will use the same construction as in case \ref{case:5<p equiv 1 mod p}, except that we fix $g=1$ and change $b_{K, i}$ such that
		\begin{equation*}
			\frac{r_{K}}{2} = \sum_{i=1}^\infty b_{K,i}2^{-i}.
		\end{equation*}
		In terms of $I_a$, the case $a=p-1$ is more important than the case $a<4$, so we read the construction as $I_1=\{1\}$ for $p=2$.
		Note that all arguments of the proof for case \ref{case:5<p equiv 1 mod p} until equation \eqref{eq:mup(limsup Apkq)} remain valid.
		Since $(1+2^i)^{- 1} \equiv 1 + 2^i\mod 2^{i+1}$, $b_{K,i}\in\{0,1\}$, and $x_{K}=0$, we find
		\begin{align*}
			\mu_2\left(\limsup_{q\to\infty} \fA_{2^K q}^2\right) &
			= 2^{-K}\sum_{i=1}^\infty b_{K, i} \mu_2 (\pm (1+2^i) + 2^{i+1}\bZ_2)
			\\&
			= 2^{-K}\sum_{i=1}^\infty b_{K, i} 2\cdot 2^{-i-1}
			= 2^{-K} \frac{r_{K}}{2},
			\\&
			= 2^{-K-1} \sum_{l=K+1}^\infty x_l 2^{K-l}
			= \sum_{l=K}^\infty x_l 2^{-l-1},
		\end{align*}
		which completes the proof in this case.
		
		\case{$p=3,5$}\label{case:p=3,5}
		We use the same construction as in case \ref{case:5<p equiv 3 mod p}, except that we further change $r_k$ and $\psi$ so that
		\begin{align*}
			r_k &= \begin{cases}
				x_{K} - 2\left\lfloor\frac{x_{K}}{2}\right\rfloor  + \sum_{l=K+2}^{\infty}x_l p^{K-l} &\text{if } k=K,	\\
				x_k - 2\left\lfloor\frac{x_k}{2}\right\rfloor &\text{if } k \ne K
			\end{cases}
			\linebreakLong
			\psi(n)&:= \begin{cases}
				f_k(q)	&\text{if } n= p^k q,\ k\le K+1,	\\
				0	&\text{otherwise}.
			\end{cases}
		\end{align*}
		The value of $g$ will also be fixed, depending on $p$.
		Note that the arguments of case \ref{case:5<p equiv 3 mod p} (which follows the arguments of case \ref{case:5<p equiv 1 mod p}) remain valid all the way to inclusion \eqref{eq:Kpk0q} and that equation \eqref{eq:Apkq} now also holds for $k=K+1$.
		By following the same argument as in case \ref{case:5<p equiv 1 mod p}, we find
		\begin{align*}
			\mu_p(\fK^p\cap p^{K}\bZ_p) &= \mu_p\left( \limsup_{q\to\infty} (\fK^p_{p^{K} q}\cup \fK^p_{p^{K+1} q}) \cap p^{K}\bZ_p \right)
		\end{align*}
		By inclusion \eqref{eq:Kpk0q}, it then follows that
		\begin{align*}
			\mu_p(\fK^p\cap p^{K}\bZ_p) &= \mu_p\left( \limsup_{q\to\infty} \big(\fA^p_{p^{K} q}\cup \fA^p_{p^{K+1} q}\big) \right)
			\linebreakLong&
			= \mu_p\left( \limsup_{q\to\infty} \fA^p_{p^{K} q} \right) + \mu_p\left( \limsup_{q\to\infty} \fA^p_{p^{K+1} q} \right)
			\linebreakLong&
			= \mu_p(\fA^p\cap p^{K}\bZ_p),
		\end{align*}
		Recalling $x_k=p-1$ for $k<K$, inclusion \eqref{eq:Apn,k<K} implies
		\begin{equation*}
			\begin{aligned}
				\mu_p(\fK^p) = \mu_p(\fA^p)
			= &\sum_{k=0}^{K-1}x_k p^{-k-1} 
			+ \mu_p\left( \limsup_{q\to\infty} \fA^p_{p^{K} q} \right) 
			\\&
			+ \mu_p\left( \limsup_{q\to\infty} \fA^p_{p^{K+1} q} \right),
			\end{aligned}
		\end{equation*}
		so that we are left to prove
		\begin{equation*}
			\mu_p\left( \limsup_{q\to\infty} \fA^p_{p^{K} q} \right) + \mu_p\left( \limsup_{q\to\infty} \fA^p_{p^{K+1} q} \right) 
			= \sum_{k=K}^\infty x_k p^{-k-1}.
		\end{equation*}
		Note that equation \eqref{eq:limsupApkq disjoint} is preserved and that it remains valid for $K$ replaced by $K+1$.
		Let $k\in\{K,K+1\}$.
		Then
		\begin{align*}
			 \bigcup_{m\in I_{x_k}}\pm m p^k +p^{k+1}\bZ_p 
			&= p^k\begin{cases}\displaystyle
				\bZ_p^\times	&\text{if } x_k=p-1,	\\\displaystyle
				\left(\pm 1 + p\bZ_p\right)	&\text{if } 2\le x_k<p-1,	\\\displaystyle
				\emptyset	&\text{otherwise}.
			\end{cases}
		\end{align*}
		From this follows
		\begin{align*}
			\mu_p(\limsup_{q\to\infty} \fA^p_{p^k q}) =& 
			p^{-k-1} 2 \left\lfloor \frac{x_k}{2}\right\rfloor
			\\&
			+ p^{-k} \sum_{i=1}^\infty \mu_p\left( \bigcup_{b'=1}^{b_{k,i}} (\pm (g+b'p^i)^{\pm 1} + p^{i+1}\bZ_p) \right).
		\end{align*}
		We are now done if we, for each of $p=3,5$, can show
		\begin{equation}
			\label{eq:sum=rk, p<=5}
			\sum_{i=1}^\infty \mu_p\left( \bigcup_{b'=1}^{b_{k,i}} (\pm (g+b'p^i)^{\pm 1} + p^{i+1}\bZ_p) \right)
			= r_k,
		\end{equation}
		as this would imply
		\begin{align*}
			\mu_p(\limsup_{q\to\infty} \fA^p_{p^k q}) &= p^{-k-1} 2 \left\lfloor \frac{x_k}{2}\right\rfloor + p^{-k} r_k
			\\&
			=\begin{cases}
				x_{K}p^{-K} + \sum_{l=K+2}^{\infty} x_l p^{-l}	&\text{if } k=K,	\\
				x_{K+1}p^{-(K+1)} 	&\text{if } k=K+1.
			\end{cases}
		\end{align*}
	
		For $p=3$, fix $g=2$.
		Then $\liminf_{l\to\infty} x_l < 2$, and we have
		\begin{equation*}
			\begin{aligned}
				\frac{r_k}{4} \le \frac{1}{4}\left(1 + \sum_{l=k+2}^{\infty} x_l 3^{k - l}\right)
			< \frac{1}{4} \left(1 + \frac{1}{3} \right) = \frac{1}{3},
			\end{aligned}
		\end{equation*}
		i.e., $b_{1,k}=0$.
		For $i\ge 2$ and $1\le b'\le b_{k,i}$, notice that
		\begin{align*}
			(2+b'3^i) + 3^{i+1}\bZ_3 &\subseteq 2 + 3^2 \bZ_3,
			\linebreakLong
			-(2+b'3^i) + 3^{i+1}\bZ_3 &\subseteq 7 + 3^2 \bZ_3,
			\linebreakLong
			(2+b'3^i)^{-1} + 3^{i+1}\bZ_3 &\subseteq 5 + 3^2 \bZ_3,
			\linebreakLong
			-(2+b'3^i)^{-1} + 3^{i+1}\bZ_3 &\subseteq 4 + 3^2 \bZ_3.
		\end{align*}
		Since all balls on the right-hand-side of the above four inclusions are disjoint, the same holds for the four balls defining $(\pm 2^{\pm 1}+b'3^i + 3^{i+1}\bZ_3)$.
		When we then vary $i$, we note that the sets
		$$(\pm 2+b'3^i)^{\pm 1} + 3^{i+1}\bZ_3\subseteq \pm 2^{\pm 1} + 3^i\bZ_3^\times$$
		are also disjoint.
		We conclude equation \eqref{eq:sum=rk, p<=5} by calculating
		\begin{align*}
			\sum_{i=1}^\infty \mu_3\bigg( \bigcup_{b'=1}^{b_{k,i}} (\pm (2+b'3^i)^{\pm 1} + 3^{i+1}\bZ_3) \bigg)
			&
			= \sum_{i=1}^\infty \sum_{b'=1}^{b_{k,i}} 4\mu_3(3^{i+1}\bZ_3)
			\\&
			= 4\sum_{i=1}^\infty b_{k,i} 3^{-i-1}
			= r_k.
		\end{align*}
		
		For $p=5$, fix $g=3$ and estimate
		\begin{equation*}
			\frac{r_k}{4} \le \frac{1}{4}\left(1 + \sum_{l=K+2}^{\infty} x_l 5^{x_{K} - l}\right)
			< \frac{1}{4}\left(1 + \frac{1}{5}\right) < \frac{2}{5},
		\end{equation*}
		so that $b_{k,1}\in\{0,1\}$.
		For $i\ge 1$ and $1\le b'\le b_{k,i}$, let $a_{i}\in\{0,1\}$ such that $a_{1} = b'$ and $a_{i}=0$ for $i>1$.
		We then find
		\begin{align*}
			(3+b'5^i) + 5^{i+1}\bZ_5 &\subseteq 3 +a_i 5 + 5^2 \bZ_5,
			\linebreakLong
			-(3+b'5^i) + 5^{i+1}\bZ_5 &\subseteq 2 + (4-a_i) 5 + 5^2 \bZ_5,
			\linebreakLong
			(3+b'5^i)^{-1} + 5^{i+1}\bZ_5 &\subseteq 2+ (3+a_i)5 + 5^2 \bZ_5,
			\linebreakLong
			-(3+b'5^i)^{-1} + 5^{i+1}\bZ_5 &\subseteq 3 + (1-a_i)5 + 5^2 \bZ_5.
		\end{align*}
		We are again left with four disjoint balls on the right-hand-side, for $i$ fixed.
		We then apply arguments in parallel to those for $p=3$ and conclude equation \eqref{eq:sum=rk, p<=5}.
		This completes the proof.
\end{proof}
In cases \ref{case:p=2} and \ref{case:p=3,5} of the above proof, note that the choice of generator $g$ matters;
for $p=2$, $g>1$ would lead to $\pm g^{\pm 1}$ representing four unique values modulo $2^i$ for sufficiently large $i$, where we want exactly 2 unique values.
For $p=3$, $g\equiv-1\mod 9$ would on the other hand lead to $\pm g^{\pm 1}$ representing only two unique values modulo $9$, where we want exactly 4.
The same issue arises modulo $25$ for $p=5$ when $g+a_{1}5=g+5\equiv \pm 7\mod 25$.

Note also that the alterations introduced in case \ref{case:p=3,5} would not be sufficient for $p=5$ if they were to be applied to case \ref{case:5<p equiv 1 mod p}, even though $5\equiv 1\mod 4$, since it would allow any $b_{k,1}$ between 0 and 4, which would lead to $g+a_{1}5\equiv \pm 7\mod 25$ for some $a_1$, regardless of $g$.

Finally, note that the construction in case \ref{case:p=3,5} would also work for $p=2$ by putting $g=3$ and $r_k=0$ for $k\ne K$, though that would not actually shorten the proof as we would then have to give $p=2$ the same amount of special attention as we gave each of $p=3$ and $p=5$.
This suggests that we are in a peculiar case of $p=2$ \textit{not} being the most troublesome prime, as that title appears to go to $p=5$, with $p=3$ as a close second.

	\section{Concluding Remarks}
	Considering how Proposition \ref{Prop:0-1 law} acts as a shell-wise $\cC^p$-variant of the 0-1 law on $\cA^p$ from \cite{Haynes}, it appears rather plausible that the $p$-adic Duffin--Schaeffer theorem 
	should also have a shell-wise $\cC^p$-variant, as formally stated below. 
	\begin{conj}
		Let $\psi:\bN\to\bR_{\geq 0}$, and let $p$ be a prime. 
		Suppose $\supp\psi \subseteq p^k\bN\setminus p^{k+1}\bN$ for some $k\in\bN_0$.
		Then
		\begin{equation*}
		\mu_p\big(\cC^p\cap p^k \bZ_p^\times \big) = \begin{cases}
		(p-1)/p^{k+1}	&\text{if } \sum_{n=1}^\infty \mu_p(\cC^p_n) = \infty,	\\
		0	&\text{if } \sum_{n=1}^\infty \mu_p(\cC^p_n) < \infty.
		\end{cases}
		\end{equation*}
	\end{conj}
	As with the original and $p$-adic Duffin--Schaeffer Theorems, the Borel-Cantelli Lemma directly implies $\mu_p(\cC^p)=0$ when the series converges.
	If the conjecture holds true, it combines with Proposition \ref{Prop:mup Cp as sum} to provide an explicit formula for determining the $p$-adic measures of $\cB^p$ and $\cC^p$.
	It is expected that the conjecture will follow from a modification of the proof of \cite[Theorem 2]{Haynes} combined with \cite[Proposition 5.4]{Koukoulopoulos+Maynard}, following the structure presented in \cite{Kristensen+Laursen}.
	In the light of \cite[Theorem 2]{Kristensen+Laursen}, it appears only natural if the measures $\mu_p(\cC^p_n)$ in the divergence criterion may be replaced by the fractions $\phi(n)\psi(n)/n$ from the real Duffin--Schaeffer Theorem, by following a similar argument.

	As for the set $\fA^p$, one might try to modify the construction with the aim of decreasing the spectrum of possible measure values, similarly to what Haynes \cite{Haynes} achieved in constructing $\cA^p$ instead of $\cB^p$.
	In his construction, Haynes effectively removed the sets $\cB_{n}^p$ that were restricted to specific shells $p^k\bZ_p^\times$ from consideration as he indirectly forced the sets $\cA_{n}^p$ with $p\mid n$ to be either empty or full \cite[Lemma 2]{Haynes}.
	Trying to get a similar modification of $\fA^p$, we might consider the set
	\begin{equation*}
		\limsup_{n\to\infty} \bigcup_{\substack{d\mid n \\ \gcd(d,pn/d)=1}} B_{\bQ_p}\left(\frac{d}{n/d}, \frac{\psi(n)}{n}\right).
	\end{equation*}
	However, for $p\ne 3,5$ and $x\in[0,(p-1)/p]\cup\{1\}$, the $\psi$ constructed in the proof of Theorem \ref{thm:spectrum var} will still produce measure $x$.
	For $p=3,5$, $\psi$ only works for $x\in \{1,(p-1)/p\}\cup\bigcup_{x_0=0}^{p-2}\frac{x_0}{p} + [0,p^{-2})$, but it seems reasonable that there should also exist a $\psi$ for the remaining $x\in[0,(p-1)/p]\cup\{1\}$; perhaps a hybrid between the constructions from cases \ref{case:p=2} and \ref{case:p=3,5} would do the trick.
	Note that this attack would work identically if the above modification were carried out on $\fK^p$ instead.
	As such, there does not seem to be any immediate `correction' to the set $\fA^p$ that would make it satisfy a 0-1 law in general.

	\vspace{6pt}
		\paragraph{\textbf{Acknowledgements}} I thank my supervisor Simon Kristensen for aiding me in writing this paper,
		and the Independent Research Fund Denmark for funding my research (Grant ref. 1026-00081B).
		I also thank the referee for providing helpful comments.

\bibliographystyle{mscplain}
\bibliography{SpectrumBib.bib}{}

\end{document}